\documentclass[a4paper,10pt]{amsart}

\usepackage{amssymb}
\usepackage{latexsym}
\usepackage{amsfonts}
\usepackage{amsfonts}
\usepackage{amsmath}
\usepackage{amsthm}
\usepackage{amsxtra}

\newcommand{\R}{\mathbb{R}}

\newcommand{\sign}{\mathop\mathrm{sign}\nolimits}
\newcommand{\cl}[1]{\overline{#1}}

\newcommand{\D}{\partial}
\newcommand{\X}{\times} 
\renewcommand{\t}[1]{\widetilde{#1}}
\newcommand{\h}[1]{\widehat{#1}}
\newcommand{\su}[1]{\overline{#1}} 

\newcommand{\rank}{\mathop\mathrm{rank}\nolimits}
\newcommand{\im}{\mathop\mathrm{im}\nolimits}

\newcommand{\rkrs}{\mathbb{R}^k\times\mathbb{R}^s}
\newcommand{\rrrn}{\mathbb{R}^r\times\mathbb{R}^{n-r}}
\newcommand{\burkrs}{ BU((-\infty, 0],\mathbb{R}^k\times\mathbb{R}^s)} 
\newcommand{\bum}{BU\big((-\infty, 0], M\big)} 
\newcommand{\buN}{BU\big((-\infty, 0], N\big)}
\newcommand{\bun}{BU\big((-\infty, 0], \R^n\big)}

\newcommand{\ctrkrs}{C_T (\R^k\times\R^s)}
\newcommand{\ctrrrn}{C_T (\R^r\times\R^{n-r})}

\newcommand{\hipH}{\textrm{(\textbf{H})}}
\newcommand{\hipK}{\textrm{(\textbf{K})}}
\newcommand{\xbf}{\boldsymbol{\mathrm{x}}}
\newcommand{\pbf}{\boldsymbol{\mathrm{p}}}

\newtheorem{definition}{Definition}[section]
\newtheorem{theorem}{Theorem}[section]
\newtheorem{corollary}[theorem]{Corollary}
\newtheorem{lemma}[theorem]{Lemma}
\newtheorem{example}[theorem]{Example}
\newtheorem{proposition}[theorem]{Proposition}
\newtheorem{remark}[theorem]{Remark}

\numberwithin{equation}{section}

\title{On a class of differential-algebraic equations with infinite delay}

\author{Luca Bisconti}
\author{Marco Spadini}
\address[L.\ Bisconti, M.\ Spadini]{Dipartimento di Sistemi e Informatica, Universit\`a di 
Firenze, Via Santa Marta 3, 50139 Firenze, Italy}

\date{\today}

\begin{document}

\begin{abstract}
We study the set of $T$-periodic solutions of a class of $T$-periodically perturbed 
Differential-Algebraic Equations, allowing the perturbation to contain a distributed and possibly 
infinite delay. Under suitable assumptions, the perturbed equations are equivalent to Retarded 
Functional (Ordinary) Differential Equations on a manifold. Our study is based on known results 
about the latter class of equations and on a ``reduction'' formula for the degree of a tangent 
vector field to implicitly defined differentiable manifolds.
\end{abstract}

\maketitle
\section {Introduction}
This paper is devoted to the study of some properties of the set of harmonic solutions to
retarded functional periodic perturbations of Differential-Algebraic Equations (DAEs) of
a particular type. The results we obtain are mainly related, on one hand, with those of 
\cite{fupespa2} concerning the method used to deal with distributed and possibly 
infinite delay and, on the other hand, with \cite{cala,spaDAE} as regards the treatment of
DAEs.

Roughly speaking, our strategy consists of reducing the perturbed DAEs that we consider to
Retarded Functional Differential Equations (RFDEs) on an implicitly-defined differentiable
manifold to which we apply the results of \cite{fupespa2}. This approach, as it is, involves 
the computation of the topological degree of a possibly complicated tangent vector field. 
This potential awkwardness is taken care of by the means of a formula of \cite[Th.\ 4.1]{spaDAE} 
(Equation \eqref{valasso} below) that allows us to replace this computation with the
more straightforward one of (essentially) the Brouwer degree of a map constructed explicitly
out of the equation.

Let $g\colon\R^k\times\R^s\to\R^s$ and $f\colon\R^k\times\R^s\to\R^k$ be given. Assume
$f$ continuous and $g\in C^{\infty}(\rkrs, \R\sp{s})$ has the property that $\D_2 g (p,q)$, 
the partial derivative of $g$ with respect to the second variable, is invertible for any 
$(p,q)\in\R^k\times\R^s\cong\R^n$. We consider the following DAE in semi-explicit form:
\begin{equation} \label{RFDAEs:0} 
   \left\{ \begin{array}{l}
                  \dot x = f(x, y), \\
                  g(x, y) = 0, \\
    \end{array} \right.
\end{equation}
and perturb it as follows:
\begin{equation} \label{RFDAEs:01} \left\{ \begin{array}{l}
      \dot  x(t) = f\big(x(t), y(t)\big) + \lambda h(t, x_t, y_t),\quad \lambda\geq 0,\\
      g\big(x(t), y(t)\big)=0, \\
    \end{array} \right.
\end{equation}
where $h\colon\R\X\burkrs\to\R^k$ is continuous and $T$-periodic, $T>0$ given, in the first 
variable. Here $\burkrs$ denotes the space of bounded uniformly continuous $(\rkrs)$-valued 
maps of $(-\infty,0]$. Here, as we will do for the remainder of the paper, we have used the 
following notation: let $\zeta\colon I\to\R^d$ be a function with $I\subseteq\R$ an interval 
such that $\inf I=-\infty$, and let $t\in I$. By $\zeta_t\colon(-\infty,0]\to\R^d$, $d>0$, 
we mean the function defined by $\theta\mapsto\zeta_t(\theta)=\zeta(t +\theta)$. According 
to this notation, $(x_t,y_t)$ is a map of $(-\infty,0]$ to $\rkrs$.

\smallskip
The resulting Equation \eqref{RFDAEs:01} is an example of \emph{Retarded Functional 
Differential-Algebraic Equation} (RFDAE). For $\lambda\geq 0$, we are interested in the 
$T$-periodic solutions of \eqref{RFDAEs:01}.
  
Since $\D_2 g (p,q)$ is invertible for any $(p, q)\in\R^k\times \R^s$, $0\in\R^s$ is a 
regular value of $g$, and so $M:=g^{-1}(0)$ is a $C^\infty$ manifold and a closed subset of 
$\R^k\X\R^s\cong\R^n$. This is important as we wish to use the results of \cite{fupespa2} 
that depend in an essential manner on $M$ being closed. \emph{Throughout the paper we will 
always denote the submanifold $g^{-1}(0)$ of $\rkrs$ by $M$}. Unless differently stated, 
the points of $M$ will written as pairs $(p,q)\in M$.

\smallskip
Notice that the Implicit Function Theorem implies that $M$ can be locally represented as a graph 
of some map from an open subset of $\R^k$ to $\R^s$. Thus, in principle, Equation \eqref{RFDAEs:01} 
can be locally decoupled. Globally, however, this might be not the case or it could not be 
convenient to do so (see, e.g.\ \cite{cala, spaDAE}).
\smallskip

As we shall see, proceeding as in \cite[\S 4.5]{dae} (compare also \cite{spaDAE}) when $\D_2 g(p,q)$ 
is invertible for all $(p,q)\in\rkrs$, Equation \eqref{RFDAEs:01} is equivalent to an RFDE on $M$ of 
the form considered in \cite{fupespa2}. Some related ideas, in the context of constrained mechanical 
systems, can be found in \cite{RR}. In order to obtain information about the set of $T$-periodic 
solutions of \eqref{RFDAEs:01}, we will use the techniques of \cite{fupespa2} combined with a result 
of \cite{spaDAE} about the degree of the tangent vector field on $M$ induced by the unperturbed 
Equation \eqref{RFDAEs:0}. Our aim will be to show the existence of a ``noncompact branch'' of 
$T$-periodic solutions of \eqref{RFDAEs:01} emanating from the set of the constant solutions of 
\eqref{RFDAEs:0}. Namely, denoted by $C_T(\rkrs)$ the Banach space of the $T$-periodic, 
$(\rkrs)$-valued functions, we will prove the existence of a connected set of triples 
$(\lambda,x,y)\in [0,+\infty)\times C_T(\rkrs)$, with $(x,y)$ a nonconstant $T$-periodic solution to 
\eqref{RFDAEs:01}, whose closure is noncompact and meets the set of constant solutions of 
\eqref{RFDAEs:0}.

In the last section of this paper, in order to illustrate our results we provide some applications
to a particular class of implicit retarded functional differential equations.

\section{Associated vector Fields and RFDEs on $M$}\label{sec2}

In this section, following \cite[Chapter 4, \S 5]{dae} (compare also \cite{spaDAE}), we associate 
to \eqref{RFDAEs:01} a RFDE on $M = g^{-1}(0)$.  

We first discuss the notion of solution to a retarded functional DAE of the form \eqref{RFDAEs:01}.
Let $f\colon\rkrs\to\R\sp{k}$ and $g\colon\rkrs\to\R\sp{s}$ be given maps with $f$ continuous 
and $g\in C^{\infty}(\rkrs, \R\sp{s})$ with the property that $\D_2 g (p,q)$ is invertible for 
any $(p,q)\in\R^k\times\R^s$.  Given $T>0$, consider also a continuous and $T$-periodic in the 
first variable  map $h\colon\R\times\burkrs\to\R\sp{k}$.  A solution of \eqref{RFDAEs:01}, for a 
given $\lambda\geq 0$, consists of a pair of functions $(x,y)\in C(I,\rkrs)$, $I\subseteq\R$ an 
interval with $\inf I=-\infty$, such that $x$ and $y$ are bounded and uniformly continuous on any 
half-line of the form $(-\infty, b]$ with $b\leq\sup I$, and
\begin{subequations}\label{daedef}
 \begin{equation}\label{daemanif}
  g\big(x(t), y(t)\big)=0,\quad\text{ for all $t\in I$},
 \end{equation}
and, eventually,
\begin{equation}\label{daeeq}
      \dot  x(t) = f\big(x(t), y(t)\big) + \lambda h(t, x_t, y_t).
\end{equation}
\end{subequations}
The latter assertion means that there exists a subinterval $J\subseteq I$ with $\sup J=\sup I$ on 
which \eqref{daeeq} holds.
Observe that, by the Implicit Function Theorem, $y$ is $C^1$ on $J$. Therefore, a solution of 
\eqref{RFDAEs:01} is a function $\zeta:=(x,y)$ which is bounded and uniformly continuous on any 
half-line of the form $(-\infty, b]$ with $b\leq\sup I$, and is eventually a $C^1$ function, i.e., 
$\zeta\in C^1(J,\rkrs)$.

Let us now associate tangent vector fields on $M$ to $f$ and $h$. Recall that given a differentiable
manifold $N\subseteq\R^n$, a continuous map $w\colon N\to\R^n$ with the property that for any $p\in N$, 
$w(p)$ belongs to the tangent space $T_pN$ to $N$ at $p$ is called a \emph{tangent vector field on 
$N$}. Similarly, a time-dependent \emph{functional (tangent vector) field on $N$} is a map 
$W\colon\R\X\buN\to\rkrs$, such that $W(t,\varphi,\psi)\in T_{(\varphi(0),\psi(0))}N$, for all 
$(t,\varphi,\psi)\in\R\X\buN$.

Consider the maps 
$\Psi\colon M\to\rkrs$ and $\Upsilon \colon \R\X\bum \to \rkrs$ defined as follows:
\begin{subequations}\label{campiv}
\begin{align}
  &\Psi (p, q)= \big( f(p, q), [\D_2 g(p, q)]^{-1}
  \D_1 g  (p, q) f(p, q) \big),\,\, \textrm{and} \label{RFDAEs:03}   \\
  &\Upsilon (t, \varphi, \psi)= \big( h(t, \varphi, \psi), -[\D_2 g(
  \varphi(0), \psi(0))]^{-1} \D_1 g (\varphi(0), \psi(0)) h(t,
  \varphi, \psi) \big). \label{RFDAEs:04}
\end{align}
\end{subequations}
Using the fact that, given a point $(p,q)\in M$, $T_{(p,q)}M$ is the kernel $\ker d_{(p,q)}g$ of 
the differential $d_{(p,q)}g$ of $g$ at $(p, q)$, it can be easily proved that $\Psi$ is tangent to 
$M$ in the sense that $\Psi(p,q)$ belongs to $T_{(p,q)}M$ for all $(p,q)\in M$ (compare, e.g.\ 
\cite{spaDAE}). Similarly, we have that $\Upsilon$ is tangent to $M$, in the sense that 
$\Upsilon(t,\varphi,\psi)\in T_{(\varphi(0),\psi(0))}M$, for all $(t,\varphi,\psi)\in\R\X\bum$. In 
other words, we see that $\Psi$ is a tangent vector field, whereas $\Upsilon$ is a time-dependent 
functional field on $M$. Since $h$ is assumed $T$-periodic in the first variable, so is $\Upsilon$.
Notice that, for any $\lambda\geq0$, the map of $\R\times BU\big((-\infty, 0],M)$ in $\rkrs$, defined 
by
\begin{equation*}
  (t,\varphi,\psi)\mapsto \Psi\big(\varphi(0), \psi(0)\big)+\lambda 
  \Upsilon(t,\varphi, \psi), 
\end{equation*}
is a functional tangent vector field as well.

We claim that \eqref{RFDAEs:01} is equivalent to the following RFDE on $M$, which keeps implicitly 
account of the algebraic condition $g(x,y)=0$:
\begin{equation} \label{RFDAEs:05} 
\dot \zeta(t) = \Psi \big(\zeta(t)\big) + \lambda  \Upsilon (t, \zeta_t),
\end{equation}
where we have used the compact notation $\zeta_t = (x_t, y_t)$, in the sense that $\zeta=(x,y)$ is 
a solution of \eqref{RFDAEs:05} in an interval $I\subseteq\R$ if and only if so is $(x,y)$ for 
\eqref{RFDAEs:01}. To verify the claim, let $\zeta =(x, y)$ be a solution of \eqref{RFDAEs:01}, 
defined on $I\subseteq \R$. Let $J\subseteq I$ be a subinterval where \eqref{daeeq} holds. By 
differentiation of the algebraic equation $g\big(x(t),y(t)\big)=0$, one gets
\begin{equation*}
  0=\partial_1 g(x(t),y(t))\dot x(t) + \partial_2 g(x(t),y(t))\dot y(t),
\end{equation*}
whence
\begin{equation} \label{RFDAEs:02}
    \dot y(t) = [\D_2 g(x(t), y(t))]^{-1}
    \D_1 g (x(t), y(t)) \left[ f(x(t), y(t)) + \lambda h(t, x_t, y_t)\right],
\end{equation}
when $t\in J$. Hence, the solutions of \eqref{RFDAEs:01} correspond to those of \eqref{RFDAEs:05}. 
The converse correspondence is more straightforward and follows from the fact that a solution 
$\zeta=(x,y)$ of \eqref{RFDAEs:05} defined on an interval $I$ with $\inf I=-\infty$  satisfies
identically $\big(x(t),y(t)\big)\in M$, which implies \eqref{daedef}, and eventually fulfills 
\[
\dot\zeta(t)=\Psi\big(\zeta(t)\big)+\lambda \Upsilon (t,\zeta_t),
\] 
whose first component is \eqref{daeeq}.

\medskip
We now introduce the important technical assumption \hipK{} below on the function $h$. This 
hypothesis implies a similar property, called condition \hipH{} (discussed e.g.\ in \cite{lavesp}), for 
the induced vector $\Upsilon$ on $M$ defined in \eqref{RFDAEs:04} that plays a central role in 
\cite{fupespa2}. This fact allows us to apply the methods of \cite{fupespa2} to our situation.

Throughout this paper, we shall suppose that $f$ is locally Lipschitz and that $h$ satisfies the 
following assumption \hipK{}: 
\begin{definition}
We say that $\mathcal{K}\colon\R\times BU\big((-\infty,0],\R^n\big)\to\R^d$ satisfies \hipK{} if,
given any compact subset $C$ of $\R\times BU\big((-\infty,0],\R^n\big)$, there exists $\ell\geq 0$ 
such that 
\begin{equation*}
| \mathcal{K}(t,\varphi) - \mathcal{K}(t,\psi)|_{d}
   \leq \ell\sup_{t\le 0} | \varphi(t)-\psi(t) |_n,
\end{equation*}
for all $(t,\varphi)\,,(t,\psi) \in C$. Here $|\cdot|_{n}$ and $|\cdot|_d$ represent the Euclidean 
norm in $\R^{n}$ and $\R^d$, respectively. Furthermore, we say that condition \hipK{} holds locally 
in $\R\X BU\big((-\infty,0],\R^n\big)$ if for any $(\tau,\eta)\in\R\X BU\big((-\infty,0],\R^n\big)$ 
there exists a neighborhood of $(\tau,\eta)$ in which \hipK{} holds.
\end{definition}

It can be proved (see e.g.\ \cite{lavesp}) that if a functional field on $M$ satisfies \hipH{} locally, 
then any associated initial value problem admits a unique solution. This shows, given the
equivalence of \eqref{RFDAEs:01} and \eqref{RFDAEs:05}, that if $f$ and $h$ satisfy \hipK{} then
any initial value problem associated to \eqref{RFDAEs:01} has unique initial solution. 

One could show that if \hipK{} is satisfied locally, then it is also satisfied globally. However, the 
local condition is easier to check. It holds, for instance, when $\mathcal{K}$ is $C^1$ or, more 
generally, locally Lipschitz in the second variable. 

The assumption that $h$ satisfies \hipK{} means that for any compact subset $C$ of 
$\R\X BU\big((-\infty,0],\rkrs\big)$, there exists a constant $\ell\geq 0$ such that
\begin{equation*}
| h(t,\varphi_1,\psi_1) - h(t,\varphi_2,\psi_2)|_{k}
   \leq \ell\sup_{t\le 0} \Big(| \varphi_1 (t)-\varphi_2(t) |_k +|\psi_1(t)-\psi_2(t)|_s\Big)
\end{equation*}
for all $(t,\varphi_1,\psi_1)\,,(t, \varphi_2,\psi_2) \in C$. Here $|\cdot|_k$ and $|\cdot|_s$ 
represent the Euclidean norm in  $\R^k$ and $\R^s$, respectively. Observe that if $f\colon\rkrs\to\R^k$ 
is a locally Lipschitz tangent vector field, and $h$ is a functional field satisfying \hipK{}, then for 
any $\lambda\in[0,+\infty)$ the map of $\R\times BU\big((-\infty,0],\rkrs)$ in $\R^k$, given by
\[
(t,\varphi,\psi)\mapsto f\big(\varphi(0),\psi(0)\big)+\lambda h(t,\varphi,\psi),
\]
verifies \hipK{} as well.

If $\Psi$ and $\Upsilon$ are the functional fields on $M$ defined in \eqref{campiv}, it is easy to 
see that for any $\lambda\in[0,+\infty)$ the map of $\R\times BU\big((-\infty,0],M)$ in $\rkrs$, 
given by
\[
(t,\varphi,\psi)\mapsto\Psi\big(\varphi(0),\psi(0)\big)+\lambda\Upsilon(t,\varphi,\psi),
\]
verifies the condition \hipH{} discussed in \cite{lavesp,fupespa2}.

\section{The degree of the tangent vector field $\Psi$}

In this section we introduce some basic notions about the degree of tangent vector fields on 
manifolds. Let $N\subseteq\R^n$ be a differentiable manifold. Recall that if $w\colon N\to\R^n$ is a 
tangent vector field on $N$ which is (Fr\'echet) differentiable at $p\in N$ and $w(p) = 0$, then 
the differential $d_p w \colon T_pN\to\R^n$ maps $T_pN$ into itself (see e.g.\ \cite{milnor}), 
so that, the determinant $\det\, d_p w$ of $d_p w$ is defined.  In the case when $p$ is a 
nondegenerate zero (i.e.\ $d_p w \colon T_p N\to \R^n$ is injective), $p$ is an isolated zero 
and $\det\, d_p w \ne 0$.  Let $W$ be an open subset of $N$ in which we assume $w$ admissible for 
the degree, that is we suppose the set $w^{-1}(0)\cap W$ is compact. Then, it is possible to 
associate to the pair $(w, W)$ an integer, $\deg(w, W)$, called the degree (or characteristic) of 
the vector field $w$ in $W$ (see e.g.\ \cite{fupespa,difftop}), which, roughly speaking, counts 
(algebraically) the zeros of $w$ in $W$ in the sense that when the zeros of $w$ are all 
nondegenerate, then the set $w^{-1}(0)\cap W$ is finite and
\begin{equation*} \label{RFDAEs:deg} 
\deg(w, W) = \sum_{q \in w^{-1}(0)\cap W} \sign\, \det\, d_q w.
\end{equation*}
The concept of degree of a tangent vector field is related to the classical one of Brouwer degree
(whence its name), but the former notion differs from the latter when dealing with manifolds.
In particular, the former does not need the orientation of the underlying manifolds. However, 
when $N=\R^n$, the degree of a vector field $\deg(w, W)$ is essentially the well known Brouwer 
degree of $w$ on $W$ with respect to $0$ (recall that in Euclidean spaces vector fields can be 
regarded as maps). For the main properties of the degree we refer e.g.\ to 
\cite{fupespa, difftop, milnor}.

Let now $g\colon\rkrs\to\R\sp{s}$ and $f\colon\R\times\rkrs\to\R\sp{k}$ be given maps such that $f$ 
is continuous and $g$ is $C^{\infty}$ with the property that $\partial_2 g (p, q)$ is invertible 
for all $(p,q)\in\rkrs$. Let $\Psi$ be the tangent vector field on $M= g^{-1}(0)$ given by 
\eqref{RFDAEs:03}. 

A crucial requisite for the remainder of the paper is that the degree of $\Psi$ is nonzero. The 
following consequence of \cite[Th.\ 4.1]{spaDAE} (see also \cite{CaSp}) allows us to replace
this condition with a more manageable one, at least in principle.

\begin{proposition} \label{RFDAEs:teo1}  Let $F\colon\rkrs\to\rkrs$ be given by
\[
(p, q) \mapsto \big(f(p, q), g(p, q)\big)
\]
and let $V\subseteq\rkrs$ be an open set.  Then, if either $\deg(\Psi, M\cap V)$ or $deg(F,V)$ is 
well defined, so is the other, and  
\begin{equation}\label{valasso}
   |\deg (\Psi, M\cap V)| = |\deg (F, V)|.
\end{equation}
\end{proposition}

\begin{proof}
Follows immediately from Theorem 4.1 in \cite{spaDAE} and the excision property.
\end{proof}

\section{Connected sets of $T$-periodic solutions}\label{sec4}

This section is concerned with the set of periodic solutions to
\eqref{RFDAEs:01}. As in Section \ref{sec2} we are given maps
$f\colon\rkrs\to\R^k$, $g\colon\rkrs\to\R^s$ and $h\colon \R \X
\burkrs \to \R\sp{k}$, and we assume that
\begin{enumerate}
\item $f$ is locally Lipschitz;
\item $g$ is $C^\infty$ and such that $\det \partial_2g(p, q) \ne 0$
  for all $(p, q)\in\rkrs$;
\item $h$ satisfies \hipK{} and, given $T > 0$, is $T$-periodic with
  respect to its first variable.
\end{enumerate}

Denote by $C_T(\rkrs)$ the Banach space of all the continuous $T$-periodic functions assuming 
values in $\rkrs$ with the usual supremum norm. We say that $(\mu, \xi)\in[0,+\infty)\X\ctrkrs$ 
is a \emph{$T$-periodic pair} for \eqref{RFDAEs:01} if $\xi=(x, y)$ satisfies \eqref{RFDAEs:01} 
for $\lambda = \mu$. Here, as well as in what follows, the elements of $\ctrkrs$ will be written 
as pairs whenever convenient. In this way, $T$-periodic pairs actually will be often written as 
triples. Moreover, given $(p,q)\in\rkrs$, denote by $\big(\su p,\su q\big)$ the element of 
$\ctrkrs$ that is constantly equal to $(p, q)$. A $T$-periodic pair of the form $(0,\su p,\su q)$ 
is called \emph{trivial}.

Let $F\colon\rkrs\to\rkrs$ be the vector field given by
\begin{equation}\label{RFDAEs:F}
  F(p, q) = \big(f(p, q), g(p, q)\big).
\end{equation}
It can be easily verified that $(\su p,\su q)$ is a constant solution
of \eqref{RFDAEs:01} for $\lambda=0$ if and only if $F(p,q)=(0,
0)$. Thus, with the above notation, the set of trivial $T$-periodic
pairs can be written as
\begin{equation*}
  \big\{ (0,\su p,\su q) \in [0,+\infty) \X \ctrkrs : F(p, q) = (0, 0)\big\}.
\end{equation*}
The following convention is very handy. Given subsets $\Omega$ and $X$
of $[0,+\infty)\X\ctrkrs$ and of $\rkrs$, respectively, with
$\Omega\cap X$ we denote the set of points of $X$ that, regarded as
constant functions, lie in $\Omega$. Namely,
\[
\Omega\cap X=\{(p,q)\in X: (0,\su p,\su q)\in\Omega\}.
\]

The next result provides an insight into the topological structure of
the set of $T$-periodic solutions of \eqref{RFDAEs:01}.

\begin{theorem}\label{RFDAEs:prop1}
  Let $f$, $h$ and $g$ be as above.  Let also $F\colon \rkrs\to\rkrs$ be
  defined as in \eqref{RFDAEs:F}.  Let $\Omega$ be an open subset of
  $[0,+\infty)\X \ctrkrs$ and Assume that $\deg\big(F,
  \Omega\cap(\rkrs) \big)$ is well-defined and nonzero. Then, the set
  of nontrivial $T$-periodic pairs of \eqref{RFDAEs:01}, admits a
  connected subset whose closure in $\Omega$ is noncompact and meets
  the set of trivial $T$-periodic pairs in $\Omega$, i.e.\ the set
  $\big\{ (0,\su p,\su q)\in\Omega:F(p,q)=(0,0)\big\}$.
  In particular, the set of $T$-periodic pairs for \eqref{RFDAEs:01}
  contains a connected component that meets $\big\{(0,\su p,\su
  q)\in\Omega:F(p,q)=(0,0)\big\}$ and whose intersection with $\Omega$
  is not compact.
\end{theorem}

\begin{proof}
Let $\Psi$ and $\Upsilon$ be as in \eqref{campiv}. Then \eqref{RFDAEs:01} is equivalent 
to \eqref{RFDAEs:05} on $M=g^{-1}(0)$.  Denote by $C_T(M)$ the metric subspace of the 
Banach space $C_T(\rkrs)$, of all the continuous $T$-periodic functions taking values in 
$M$. Let also $\mathcal{O}$ be the open subset of $[0,+\infty)\X C_T(M)$ given by
\[
  \mathcal{O}=\Omega\cap \Big([0,+\infty)\X C_T(M)\Big).
\]
Given $Y\subseteq M$, by $\mathcal{O}\cap Y$ we mean the set of all those points of $Y$ that, 
regarded as constant functions, lie in $\mathcal{O}$. With this convention one clearly has 
$\Omega\cap Y=\mathcal{O}\cap Y$ and, in particular, $\Omega\cap M=\mathcal{O}\cap M$. This 
identity and Proposition \ref{RFDAEs:teo1} imply that
\[
  \deg(\Psi,\mathcal{O}\cap M)=\deg(\Psi,\Omega\cap
         M)=\pm\deg\big(F,\Omega\cap(\rkrs)\big)\neq 0.
\]
Thus, Theorem 4.1 in \cite{fupespa2} yields the existence of a connected subset $\Lambda$ of
\[
  \big\{ (\lambda,x,y)\in \mathcal{O}:\text{$(x,y)$ is a nonconstant
    solution of \eqref{RFDAEs:05} }\big\},
\]
  whose closure $\cl{\Lambda}$ in $\mathcal{O}$ is not compact and meets the set
\[
  \big\{(0,\su p,\su q)\in\mathcal{O}:\Psi(p,q)=(0,0)\big\},
\]
that coincides with $\big\{(0,\su p,\su q)\in\Omega:F(p,q)=(0,0)\big\}$.

Clearly, each $(\lambda,x,y)\in\Lambda$ is a nontrivial $T$-periodic pair of \eqref{RFDAEs:01}. 
Since $M$ is closed in $\rkrs$, it is not difficult to prove that any set which is closed in 
$\mathcal{O}$ is closed in $\Omega$ too, and vice versa. Thus, $\cl\Lambda$ coincides with 
the closure of $\Lambda$ in $\Omega$. The first part of the assertion follows.

Let us prove the last part of the assertion. Consider the connected component $\Gamma$ of the
set of $T$-periodic pairs that contains the connected set $\Lambda$. We shall now show that 
$\Gamma$ has the required properties.  Clearly, $\Gamma$ meets the set 
$\big\{(0,\su{p},\su{q})\in\Omega:g(p,q)=0\big\}$ because the closure of $\Lambda$ in $\Omega$ 
does. Moreover, $\Gamma\cap\Omega$ cannot be compact, since it contains the (noncompact) 
closure of $\Lambda$ in $\Omega$.
\end{proof}

\begin{remark}\label{obs1}
  Let $\Omega$ be as in Theorem \ref{RFDAEs:teo1}, and assume that
  $\Gamma$ is a connected component of $T$-periodic pairs of
  \eqref{RFDAEs:01} that meets $\{(0,\su p,\su
  q)\in\Omega:F(p,q)=(0,0)\}$ and whose intersection with $\Omega$ is
  not compact. Ascoli's Theorem implies that any bounded set of
  $T$-periodic pairs is relatively compact.  Then, the closed set
  $\Gamma$ cannot be both bounded and contained in $\Omega$. In
  particular, if $\Omega$ is bounded, then $\Gamma$ necessarily meets
  the boundary of $\Omega$.
\end{remark}

The following corollary ensures the existence of a Rabinowitz-type
branch of $T$-periodic pairs.

\begin{corollary}\label{cor1}
  Let $f$, $h$ and $g$ be as in Theorem \ref{RFDAEs:teo1}. Let
  $V\subseteq\rkrs$ be open and assume that $\deg(F,V)$ is well
  defined and nonzero. Then, there exists a connected component
  $\Gamma$ of $T$-periodic pairs of \eqref{RFDAEs:01} that meets the
  set
  \[
  \big\{(0,\su{p},\su{q})\in [0,+\infty)\times C_T(\rkrs): (p,q)\in
  V\cap F^{-1}(0,0)\big\}
  \]
  and is either unbounded or meets
  \[
  \big\{(0,\su{p},\su{q})\in [0,+\infty)\times C_T(\rkrs): (p,q)\in
  F^{-1}(0,0)\setminus V\big\}.
  \]
\end{corollary}

\begin{proof}
  Consider the open subset $\Omega$ of $[0,+\infty)\times C_T(\rkrs)$
  given by
  \begin{multline*}
    \Omega=\big([0,+\infty)\times C_T(\rkrs)\big)\setminus \\
    \setminus \big\{ (0,\su{p},\su{q})\in [0,+\infty)\times
    C_T(\rkrs): (p,q)\in F^{-1}(0,0)\setminus V\big\}.
  \end{multline*}
Clearly, we have $\Omega\cap(\rkrs)=V$. Hence $\deg\big(F,\Omega\cap(\rkrs)\big)\neq 0$. 
Theorem \ref{RFDAEs:teo1} implies the existence of a connected component $\Gamma$ of 
$T$-periodic pairs of \eqref{RFDAEs:01} that meets $\{(0,\su p,\su q)\in\Omega:F(p,q)=0\}$ 
and whose intersection with $\Omega$ is not compact. Because of Remark \ref{obs1}, if 
$\Gamma$ is bounded, then it meets the boundary of $\Omega$ which is given by
  \[
  \big\{ (0,\su{p},\su{q})\in [0,+\infty)\times C_T(\rkrs): (p,q)\in
  F^{-1}(0,0)\setminus V\big\}.
  \]
And the assertion is proved.
\end{proof}

\medskip
\begin{example}
The well-known logistic equation (see, e.g. \cite{brauer-cast})
  \begin{equation}\label{RFDAEs:logisteq}
    \dot x=\alpha x -\beta x^2
  \end{equation}
is sometimes used as a model for a population $x$ with birth and mortality rate $\alpha x$ and 
$\beta x^2$, respectively. Consider a generalization of \eqref{RFDAEs:logisteq} where the mortality 
rate $y$ is related to the population by the implicit relation $g(x,y)=0$. This generalized model 
is expressed by the following DAE:
  \begin{equation*}
    \left\{
      \begin{array}{l}
        \dot x= \alpha x - y ,\\
        g(x,y)=0.
      \end{array}
    \right.
  \end{equation*}
If we allow the population's fertility to undergo periodic oscillations ---say $\lambda h(t,x_t)$ 
with $\lambda\geq 0$ ---depending possibly on the history of the population, the above model
can be modified into the following RFDAE:
  \begin{equation}\label{RFDAEs:genlogeq}
    \left\{
      \begin{array}{l}
        \dot x(t)= \alpha x(t) - y(t)+\lambda h(t,x_t),\\
        g\big(x(t),y(t)\big)=0.
      \end{array}
    \right.
  \end{equation}
Examples of the perturbation $h(t, x_t)$ can obtained taking inspiration from models describing 
the dynamics of animals populations (see, e.g.\ \cite{brauer-cast, burt-1}) in which the delay 
depends on time. 

Here, however, we are interested in Equation \eqref{RFDAEs:genlogeq} in itself regardless of its 
biological meaning. In particular, we wish to look at how Theorem \ref{RFDAEs:prop1} can be applied 
to it. Notice that it could be impossible to get a biologically relevant result merely 
from such an application. In fact, \eqref{RFDAEs:genlogeq} makes sense as a population model only 
as long as $x\geq 0$, but there is no guarantee that the $x$-component of all the solutions in the 
branch of $T$-periodic pairs provided by this theorem are nonnegative.

Consider, for instance, the case when $\alpha>0$ and $g(x,y)=y^3+y-x^5$. Let 
$\Omega=[0,+\infty)\X C_T(\rkrs)$. The map $F$ defined in \eqref{RFDAEs:F} is given by 
$F(x,y)=(\alpha x -y, y^3+y-x^5)$, and a simple direct computation shows that 
$\Omega\cap F^{-1}(0,0)$ consists of the singleton $\{(0,0)\}$ and that 
$\deg(F,\Omega\cap\R^2)=-1$.  Hence, Theorem \ref{RFDAEs:prop1} yields an unbounded connected 
component of periodic $2\pi$-periodic pairs emanating from the trivial $2\pi$-periodic pair 
$(0,\su 0,\su 0)$.
\end{example}

\section{An application}
This section is primarily intended as an illustration of our main result Theorem \ref{RFDAEs:prop1}. 
For this reason we will not pursue maximal generality but restrict ourselves to simple situations. 
Below, we consider retarded periodic perturbations of a particular class of implicit ordinary 
differential equations. Namely, we study equations of the following form:
\begin{equation} \label{RFDAEs:strange1} 
E \dot{\xbf}(t) =  \mathcal{F}\big(\xbf(t)\big) + \lambda \mathcal{H}(t, \xbf _t),
\qquad \lambda\geq 0,
\end{equation}
where $E\colon \R^n\to\R^n$ is a linear endomorphism of $\R^n$, $\mathcal{F}\colon \R^n\to\R^n$ 
and $\mathcal{H}\colon\R\X\bun\to\R^n$ are continuous maps with $\mathcal{F}$ locally Lipschitz 
and $\mathcal{H}$ verifies condition \hipK{}.

Equation \eqref{RFDAEs:strange1}, when $\lambda=0$, is quite a particular case of semi-linear DAE
(see e.g.\ \cite{RR0} and references therein). Such equations, even in the further particular case 
when $\mathcal{F}$ is linear, have some practical interest. In fact, they can be used to model 
such things as electrical circuits or chemical reactions (see e.g.\ \cite{Sch}). Our approach here
is inspired to that of \cite{dae,RR0} for the linear, constant coefficients, case.

We will show how, in some circumstances, \eqref{RFDAEs:strange1} can be transformed into a 
\mbox{RFDAE} of type \eqref{RFDAEs:01} by the means of relatively simple linear transformations. We 
will apply the results of the previous section to the resulting RFDAE. A first example of the above 
mentioned transformation is considered in the following remark:

\begin{remark}\label{RFDAEs:prop-nonlin-ex}
Consider Equation \eqref{RFDAEs:strange1} and let $r>0$ be the rank of $E$. Assume that there exists 
a orthogonal basis of $\R^r\X\R^{n-r}$ with respect to which $E$ can be written in the following
block form:
\begin{subequations}\label{RFDAEs:ex0}
\begin{equation}\label{RFDAEs:ex0a}
  E\simeq\begin{pmatrix} 
               E_{11} & E_{12} \\
               0      & 0 
        \end{pmatrix},\text{ with $E_{11}\in\R^{r\X r}$ invertible and 
         $E_{12}\in\R^{r\X (n-r)}$}. 
\end{equation}
Assume also that in this basis $\mathcal{H}$ has, with a slight abuse of notation, the following
form:
\begin{equation}\label{RFDAEs:ex0b}
\mathcal{H}(t, \varphi)=\begin{pmatrix}
              \mathcal{H}_1(t, \varphi) \\
               0 
                                \end{pmatrix},\; 
         \textrm{ with }\, \mathcal{H}_1\colon \R
      \times \bun\to \R^r.  
\end{equation}
  \end{subequations}
In $\R^n\simeq\R^r\X\R^{n-r}$ put $\xbf =(\xi,\eta)$ and let 
$J_E\colon\R^r\X\R^{n-r}\to\R^r\X\R^{n-r}$ be the linear transformation represented by 
the following block matrix:
\begin{equation*}
    \begin{pmatrix}
      E_{11}^{-1} & -E_{11}^{-1}E_{12} \\ 
      0           &   I
    \end{pmatrix}.
\end{equation*}
Let $(x,y)=J_E^{-1}(\xi,\eta)$, and let $\mathcal{F}_1(\xi,\eta)$ and $\mathcal{F}_2(\xi,\eta)$ 
denote the projection of $\mathcal{F}(\xi,\eta)$ onto the first and second factor, respectively, 
of $\R^r\X\R^{n-r}$. Then, in the new variables $x$ and $y$ Equation \eqref{RFDAEs:strange1} 
becomes, with a slight abuse of notation,
\[ 
EJ_E \begin{pmatrix}
      \dot x\\
      \dot y
     \end{pmatrix} = 
     \begin{pmatrix}
      \mathcal{F}_1\big(J_E(x,y)\big) \\
      \mathcal{F}_2\big(J_E(x,y)\big)
     \end{pmatrix}+\lambda
     \begin{pmatrix}
      \mathcal{H}_1\big(t,J_E( x_t,y_t)\big) \\
      0
     \end{pmatrix},
\]
or, equivalently, 
\begin{equation} \label{RFDAEs:test} 
\left\{\begin{array}{l}
        \dot{{x}} = \widetilde{\mathcal{F}}_1({x}, {y}) + \lambda
        \widetilde{\mathcal{H}}_1 (t,{x}_t, {y}_t),\\
        \widetilde{\mathcal{F}}_2({x}, {y}) =0,
      \end{array}\right.
\end{equation}
where $\t{\mathcal{F}}_i(x,y)=\mathcal{F}_i\big(J_E(x,y)\big)$, $i=1,2$, and
$\t{\mathcal{H}}_1(t,\varphi) =\mathcal{H}_1\big(t,J_E\circ\varphi\big)$, for any
$(t,\varphi)\in\R\X\bun$.  Furthermore, since $\mathcal{H}$ satisfies \hipK{}, it 
is not difficult to prove that $\t{\mathcal{H}}_1$ satisfies \hipK{} as well.
\end{remark}

\begin{example}\label{ex1app}
 Consider the following DAE in $\R^2\X\R^3$:
\begin{equation}\label{daeex}
 \left\{
\begin{array}{l}
 \dot\xi_1+\dot\xi_2+\dot\eta =\xi_2,\\
 \dot\xi_1 = -\xi_1+\xi_2^2+\eta ,\\
 0=\eta ^3+\eta +\xi_1,
\end{array}
\right.
\end{equation}
which can be written as the implicit ODE below where $\xbf=(\xi_1,\xi_2,\eta)$
\begin{equation}\label{impeqex}
 E\dot\xbf=\mathcal{F}(\xbf),
\end{equation}
where $E$ is the endomorphism of $\R^2\X\R$ represented by the block matrix
\[
\left( \begin{array}{cc|c} 
          1 & 1 & 1 \\
          1 & 0 & 0 \\
          \hline
          0 & 0 & 0 \\
\end{array} \right)
\]
and $\mathcal{F}\colon\R^3\to\R^3$ is given by 
$\mathcal{F}(\xi_1,\xi_2,\eta )=(\xi_2,-\xi_1+\xi_2^2+\eta ,\eta ^3+\eta +\xi_1)$. Let 
$J_E$ be the linear transformation of $\R^3\simeq\R^2\X\R$ represented by the block matrix
\[
 \left( \begin{array}{cc|c} 
          0 & 1 & 0 \\ 
          1 & -1 & -1 \\ 
          \hline
          0 & 0 & 1 \\
        \end{array} \right),
\]
and put $(x_1,x_2,y)=J_E^{-1}(\xi_1,\xi_2,\eta )$. One has that
\[
\mathcal{F}\big(J_E(x_1,x_2,y)\big)
      =\big( x_1 -x_2 -y, (x_1 -x_2- y)^2 +y-x_2, y^3  + y + x_2\big),
\]
As in Remark \ref{RFDAEs:prop-nonlin-ex}, for $\xi=(\xi_1,\xi_2)$, let $\mathcal{F}_1(\xi,\eta)$ 
and $\mathcal{F}_2(\xi,\eta)$ denote the projection of $\mathcal{F}(\xi,\eta)$ onto the first and 
second factor, respectively, of $\R^2\X\R$. Put also 
\[
\t{\mathcal{F}}_i(x,y)=\mathcal{F}_i\big(J_E(x,y)\big),\quad i=1,2,
\] 
where $x=(x_1,x_2)$. Proceeding as in Remark \ref{RFDAEs:prop-nonlin-ex}, we transform Equation 
\eqref{daeex} into
\[
 \left\{ \begin{array}{l} 
        \dot x = \t{\mathcal{F}}_1(x, y), \\ 
        \t{\mathcal{F}}_2(x, y) =0,
      \end{array}\right.
\]
that can be written more explicitly as follows:
\begin{equation*}
    \left\{ \begin{array}{l} 
        \dot x_1 = x_1 -x_2- y, \\ 
        \dot x_2 = (x_1 -x_2+ y)^2  +y -x_2, \\
        y^3  + y + x_2 =0.
      \end{array}\right.
  \end{equation*}
\end{example}

\medskip
Theorem \ref{RFDAEs:prop1} combined with the above Remark \ref{RFDAEs:prop-nonlin-ex},  
yields Proposition \ref{RADEs:prop01-nonlinearcase} below concerning the set of $T$-periodic 
solutions of \eqref{RFDAEs:strange1}. We use here the convention on the subsets of 
$[0,+\infty)\X\ctrrrn$ introduced in section 4. We also need to introduce some further 
notation.

A pair $(\lambda,\xbf )\in[0,+\infty)\X C_T(\R^n)$ is a \emph{$T$-periodic pair 
for \eqref{RFDAEs:strange1}} if $\xbf $ is a solution of \eqref{RFDAEs:strange1}
corresponding to $\lambda$. A $T$-periodic pair $(0,\xbf )$ for 
\eqref{RFDAEs:strange1} is \emph{trivial} if $\xbf $ is constant.

\begin{proposition} \label{RADEs:prop01-nonlinearcase} 
Consider Equation \eqref{RFDAEs:strange1} where $E\colon\R^n\to\R^n$ is linear, 
$\mathcal{F}\colon\R^n\to\R^n$ and $\mathcal{H}\colon\R\X\bun\to\R^n$ are continuous maps such 
that $\mathcal{F}$ is locally Lipschitz and $\mathcal{H}$ verifies condition \hipK{} and is 
$T$-periodic in the first variable. Assume, as in Remark \ref{RFDAEs:prop-nonlin-ex}, that $r>0$
is the rank of $E$ and that there exists an orthogonal basis of $\R^n\simeq\rrrn$ such that $E$ 
and $\mathcal{H}$ can be represented as in \eqref{RFDAEs:ex0}. Relatively to this decomposition
of $\R^2$ suppose that $\D_2{\mathcal{F}}_2(\xi,\eta)$ is invertible for all $(\xi,\eta)\in\rrrn$. 

Let $\Omega$ be an open subset of $[0,+\infty)\X C_T(\R^n)$ and suppose that 
$\deg(\mathcal{F},\Omega\cap\R^n)$ is well-defined and nonzero. Then, there exists 
a connected subset $\Gamma$ of nontrivial $T$-periodic pairs for \eqref{RFDAEs:strange1} 
whose closure in $\Omega$ is noncompact and meets the set 
$\big\{(0,\su{\pbf})\in \Omega:\mathcal{F}(\pbf) =0\big\}$.
\end{proposition}

\begin{proof}
Let $J_E$ be the linear transformation introduced in Remark \ref{RFDAEs:prop-nonlin-ex},
and consider the map $\h{J_E}\colon[0,+\infty)\X C_T(\R^n)\to [0,+\infty)\X C_T(\R^n)$
given by
\[
 \h{J_E}(\lambda,\psi)=(\lambda,J_E\circ\psi).
\]
Observe that since $J_E$ is invertible, $\h{J_E}$ is continuous and invertible, with $\h{J_E}^{-1}$ 
given by $\h{J_E}^{-1}(\lambda,\psi)=(\lambda,{J_E}^{-1}\circ\psi)$ and, hence, continuous. With the 
convention on the subsets of $[0,+\infty)\X C_T(\R^n)$ introduced in section \ref{sec4}, we have
\[
 \h{J_E}^{-1}(\Omega)\cap\R^n={J_E}^{-1}(\Omega\cap\R^n).
\]

According to Remark \ref{RFDAEs:prop-nonlin-ex}, under our assumptions Equation 
\eqref{RFDAEs:strange1} is equivalent to the RFDAE \eqref{RFDAEs:test}. We now show that
Theorem \ref{RFDAEs:prop1} can be applied to Equation \eqref{RFDAEs:test}. Define
$\t{\mathcal{F}}\colon\R^n\to\R^n$ by 
$
\t{\mathcal{F}}(\pbf)=\big(\t{\mathcal{F}}_1(\pbf)\,,
                                 \,\t{\mathcal{F}}_2(\pbf)\big)
                           =\mathcal{F}\big(J_E(\pbf)\big)
$.
The property of invariance under diffeormorphism of the degree (also called topological 
invariance, see e.g.\ \cite{fupespa}) yields
\begin{equation}
\begin{split}\label{RFDAEs:Fcorsiva}
   \deg\big(\mathcal{F},\Omega\cap\R^n\big) =&\,
    \deg\big({J_E}^{-1}\circ\mathcal{F}\circ J_E,J_{E}^{-1}(\Omega\cap\R^n)\big)\\
         &= \deg\big({J_E}^{-1}\circ\t{\mathcal{F}},\h{J_{E}}^{-1}(\Omega)\cap\R^n\big).
\end{split}
\end{equation}
Also, it is not difficult to show that
\begin{equation}\label{id2}
 \deg\big({J_E}^{-1}\circ\t{\mathcal{F}},\h{J_{E}}^{-1}(\Omega)\cap\R^n\big) =
  \sign\det(J_E) \deg\big(\t{\mathcal{F}},\h{J_{E}}^{-1}(\Omega)\cap\R^n\big),
\end{equation}
so that, being $\deg\big(\mathcal{F},\Omega\cap\R^n\big)$ nonzero by assumption,
\eqref{RFDAEs:Fcorsiva}--\eqref{id2} yield
\[
 \deg\big(\t{\mathcal{F}},\h{J_{E}}^{-1}(\Omega)\cap\R^n\big)\neq 0.
\]
Hence, Theorem \ref{RFDAEs:prop1} yields a connected set $\Xi\subseteq \h{J_E}^{-1}(\Omega)$ 
of $T$-periodic pairs of \eqref{RFDAEs:test} whose closure in $\h{{J_E}}^{-1}(\Omega)$ is 
noncompact and meets  the set 
\[
\big\{(0,\su{\pbf})\in\h{J_E}^{-1}(\Omega):\t{\mathcal{F}}(\pbf)=0\big\}.
\]
Since $\h{J_E}$ is a homeomorphism, it is not difficult to show that $\Gamma = \h{J_E}(\Xi)$ 
has the required properties.
\end{proof}

\begin{example}
Let $\mathcal{H}\colon\R\X BU\big((-\infty,0],\R^3\big)\to\R^3$ be as in \eqref{RFDAEs:ex0b} 
with $r=2$. Assume also that $\mathcal{H}$ is $T$-periodic continuous in the first variable. 
Consider the retarded perturbation $\lambda\mathcal{H}(t,\xi_t)$ of Equation \eqref{daeex} in 
Example \ref{ex1app}. Namely, 
\[
 E\dot\xbf(t)=\mathcal{F}\big(\xbf(t)\big)+\lambda \mathcal{H}(t,\xbf_t),
\]
where, using the same notation of Example \ref{ex1app}, we put $\xbf=(\xi_1,\xi_2,\eta)$. 
Take $\Omega=[0,+\infty)\X C_T(\R^2\X\R)$. Observe that $\mathcal{F}^{-1}(0,0,0)=\{(0,0,0)\}$
so that the degree of $\mathcal{F}$ in $\Omega\cap\R^3$ is well-defined. A simple computation 
shows that $\deg(\mathcal{F},\Omega\cap\R^3)=\deg(\mathcal{F},\R^3)=-1$. Thus, Proposition 
\ref{RADEs:prop01-nonlinearcase} yields the existence of a connected subset of nontrivial 
$T$-periodic pairs for the above equation whose closure in $\Omega$ is noncompact and meets 
the set 
\[
\big\{(0,\su{\pbf})\in \Omega:\mathcal{F}(\pbf) =0\big\}=
\{ (0,\su 0,\su 0,\su 0)\in \Omega\}.
\]
\end{example}
\bigskip

Observe that Proposition \ref{RADEs:prop01-nonlinearcase} seems to impose some rather 
severe constraints on the form of $E$ and $\mathcal{H}$ in Equation \eqref{RFDAEs:strange1}. 
In fact, with the help of some linear transformation, one can sometimes lift these 
restrictions. This is the case when the perturbing term $\mathcal{H}$ has a particular 
`separated variables' form that agrees with $E$ in the sense of Equation \eqref{RFDAEs:ex1a}
below. Namely, we consider the following equation:
\begin{equation} \label{RFDAEs:exx1} 
E \dot{\xbf}(t) = \mathcal{F}\big(\xbf(t)\big)  + \lambda C(t) S(\xbf _t),
\end{equation}
where $C\colon \R \to \R^{n\X n}$, $S \colon \bun\to \R^{n}$ are continuous maps, $E$ is a 
(constant) $n\X n$ matrix, and $\mathcal{F}$ is as in Equation \eqref{RFDAEs:strange1}.  We 
also assume that $C$ and $E$ agree in the following sense:
  \begin{equation}\label{RFDAEs:ex1a}
   \ker\, C^T(t) = \ker\, E^T,\;\forall \,t\in\R, \text{ and $\dim\ker\, E^T>0$},
  \end{equation}
As a consequence of the well-known Rouch\'e-Capelli Theorem we get
\begin{multline*}
n-\rank\, E=n-\rank\, E^T=\dim\ker\, E^T=\\
=\dim\ker\, C(t)^T =n-\rank\, C(t)^T=n-\rank\, C(t).
\end{multline*}
Thus, we have that
  \begin{equation}\label{RFDAEs:ex1aa}
      \rank\, E =\rank\, C(t) \text{ is constant and greater than $0$ for all $t\in\R$}.
    \end{equation}
This is a singular value decomposition (see, e.g., \cite{GoVL}) argument based on the 
following technical result from linear algebra: 

\begin{lemma} \label{RFDAEs:lemma00} 
Let $E\in\R^{n\X n}$ and $C\in C\big(\R, \R^{n\X n}\big)$ be as in \eqref{RFDAEs:ex1a}. Put
$r=\rank\, E$, and let $P, Q\in \R^{n\X n}$ be orthogonal matrices that realize a  singular 
value decomposition for $E$.  Then it follows that
  \begin{equation} \label{RFDAEs:continuousSVD} PC(t) Q^T
    = \begin{pmatrix} \widetilde{C}_{11}(t) & 0  \\
      0 & 0 \end{pmatrix},\quad \forall t\in \R ,
  \end{equation}
  with $\widetilde C_{11}\in C\big(\R, \R^{r\X r} \big)$ invertible
  for any $t\in\R$.
\end{lemma}

\medskip
We will provide a proof for Lemma \ref{RFDAEs:lemma00} for the sake of completeness but,
before doing that, we show how it can be used to convert Equation \eqref{RFDAEs:exx1} into 
\eqref{RFDAEs:strange1}.  We begin with an example.

\begin{example}
  Consider Equation \eqref{RFDAEs:exx1} with
  \begin{equation*}
    \label{RFDAEs:commutingmatrices}
E= \begin{pmatrix}
      0 & 2 & 0 & 0 \\
      1 & 0 & 0 & 0 \\
      0 & 0 & 0 & 0 \\
      0 & 0 & 0 & 1 \\
    \end{pmatrix}
\quad\textrm{ and }\quad 
C(t) = \begin{pmatrix}
      c(t) & 0 & 0 & 0 \\
      0 & c(t) & 0 & 0 \\
      0 & 0  & 0 & 0 \\
      0 & 0  & 0 & d(t) \\
    \end{pmatrix},\,\,
  \end{equation*}
where, for any $t\in \R$, $c(t)= \sin(t) + 2$ and $d(t)=\cos(t)+3$. It can be easily verified 
that, with this choice of $E$ and $C$, \eqref{RFDAEs:ex1a} is satisfied. Here, clearly, $r=3$
and $n=4$. Consider the following orthogonal matrices
  \begin{equation*}
    P =\begin{pmatrix} 
      0 & 1 & 0 & 0 \\
      1 & 0 & 0 & 0 \\
      0 & 0 & 0 & 1 \\
      0 & 0 & 1 & 0 \\
    \end{pmatrix}\,\, \textrm{and}\,\, Q =\begin{pmatrix}
      1 & 0 & 0 & 0 \\
      0 & 1 & 0 & 0 \\
      0 & 0   & 0 & 1 \\
      0 & 0   & 1 & 0 \\
    \end{pmatrix},
  \end{equation*}
that realize a singular value decomposition for $E$, that is, in block-matrix form in 
$\R^4\simeq\R^3\X\R$,
  \begin{equation*}
    PEQ^T = \left(\begin{array}{ccc|c} 
        1 & 0 & 0 & 0\\
        0 & 2 & 0& 0\\
        0 & 0 & 1 &  0 \\
        \hline
        0&  0 & 0  & 0 \\
      \end{array}\right)\,\,
    \textrm{and}\,\,
    PC(t)Q^T =\left(\begin{array}{ccc|c} 
        0    & c(t) & 0 & 0 \\
        c(t) &   0  & 0 & 0\\
        0    &   0  & d(t) & 0  \\
        \hline
        0& 0  &0  & 0 \\
      \end{array} \right). 
  \end{equation*}
Let us consider the orthogonal change of coordinates $\xbf =Q^T x$. Multiplying
\eqref{RFDAEs:exx1} by $P$ on the left we get the following equivalent equation:
\begin{equation}\label{svdeq1}
 PEQ^T\dot x(t)=P\mathcal{F}\big(Q^Tx(t)\big)+\lambda PC(t)Q^TQS(Q^Tx_t) .
\end{equation}
Set $\t E= PEQ^T$, $\t{\mathcal{F}}(x)=P\mathcal{F}(Q^Tx)$ for all $x\in\R^4$, and finally, 
put $\t{\mathcal{H}}(t,\varphi) = PC(t)Q^TQS(Q^T\varphi)$ for all 
$(t,\varphi)\in\R\X BU\big((-\infty,0],\R^4\big)$. Thus \eqref{svdeq1} can be rewritten as
\[
 \t E\dot x(t)=\t{\mathcal{F}}\big(x(t)\big)+\lambda \t{\mathcal{H}}(t,x_t).
\]
It is easily verified that $\t E$ and $\t{\mathcal{H}}$ satisfy \eqref{RFDAEs:ex0}, so that 
\eqref{svdeq1} is precisely of the form \eqref{RFDAEs:strange1}. In other words, we have 
transformed \eqref{RFDAEs:exx1}, for $E$ and $C$ as above, into an equation of the form 
considered in Proposition \ref{RADEs:prop01-nonlinearcase}.
\end{example}
  
Let us now consider Equation \eqref{RFDAEs:exx1} more in general. Let $r>0$ be the rank of 
$E$, and assume that \eqref{RFDAEs:ex1a} is satisfied. Then Lemma \ref{RFDAEs:lemma00} yields 
orthogonal matrices $P$ and $Q$ in $\R^{n\X n}$ such that, for every $t\in\R$, $PC(t)Q^T$ is 
as in \eqref{RFDAEs:continuousSVD} and realize a singular value decomposition of $E$. That is
\begin{equation}\label{RFDAEs:peq}
  PEQ^T  = \begin{pmatrix} \widetilde E_{11} 
    & 0 \\ 0 & 0 \end{pmatrix}
\end{equation}
where $\widetilde E_{11}\in \R^{r\X r}$ is a diagonal matrix with positive diagonal elements.
As in the above example, consider the orthogonal change of coordinates $\xbf =Q^T x$ in 
Equation \eqref{RFDAEs:exx1} and multiply by $P$ on the left. We get the equivalent equation
\begin{equation}\label{svdtransf}
 \t E\dot x(t)=\t{\mathcal{F}}\big(x(t)\big)+\lambda \t{\mathcal{H}}(t,x_t).
\end{equation}
where $\t E$, $\t{\mathcal{F}}$ and $\t{\mathcal{H}}$ are given by $\t E= PEQ^T$, 
$\t{\mathcal{F}}(x)=P\mathcal{F}(Q^Tx)$ for all $x\in\R^n$, and 
$\t{\mathcal{H}}(t,\varphi) = PC(t)Q^TQS(Q^T\varphi)$ for all $(t,\varphi)\in\R\X\bun$. A 
straightforward computation shows that $\t E$ and $\t{\mathcal{H}}$ satisfy conditions
\eqref{RFDAEs:ex0}. Therefore, \eqref{svdtransf} is of the form considered in Proposition 
\ref{RADEs:prop01-nonlinearcase} from which we deduce the following consequence: 

\begin{corollary}
Consider Equation \eqref{RFDAEs:exx1} where the maps $C\colon\R\to\R^{n\X n}$ and 
$S\colon\bun\to\R^n$ are continuous, $E$ is a (constant) $n\X n$ matrix, and $\mathcal{F}$ 
is such that $\mathcal{F}$ is locally Lipschitz and $S$ verifies condition \hipK{}. Suppose 
also that $C$ and $E$ satisfy \eqref{RFDAEs:ex1a} and that $C$  is $T$-periodic. Let $r>0$ be 
the rank of $E$ and assume that there exists an orthogonal basis of $\R^n\simeq\rrrn$ such that
$E$ is as in \eqref{RFDAEs:ex0}. Assume also that, relatively to this decomposition, 
$\D_2{\mathcal{F}}_2(\xi,\eta)$ is invertible for all $x=(\xi,\eta)\in\rrrn$.

Let $\Omega$ be an open subset of $[0,+\infty)\X C_T(\R^n)$ and suppose that 
$\deg(\mathcal{F},\Omega\cap\R^n)$ is well-defined and nonzero. Then, there exists 
a connected subset $\Gamma$ of nontrivial $T$-periodic pairs for \eqref{RFDAEs:exx1} 
whose closure in $\Omega$ is noncompact and meets the set 
$\big\{(0,\su{\pbf})\in \Omega:\mathcal{F}(\pbf) =0\big\}$.
\end{corollary}

\begin{proof}
Consider the map $\h Q\colon [0,+\infty)\X C_T(\R^n)\to [0,+\infty)\X C_T(\R^n)$ given
by $\h Q(\lambda,\psi)=(\lambda,Q\psi)$. Clearly, $T$-periodic pairs of \eqref{RFDAEs:exx1} 
correspond to those of \eqref{svdtransf} under $\h Q$. The invariance under diffeomorphisms 
of the degree (or topological invariance, compare e.g.\ \cite{fupespa}) implies 
\[
\deg\big(\t{\mathcal{F}}, Q(\Omega)\cap\R^n\big)\neq 0.
\]
The assertion follows immediately applying Proposition \ref{RADEs:prop01-nonlinearcase} to 
Equation \eqref{svdtransf}. 
\end{proof}

\medskip
We conclude this section with a proof of our technical Lemma.

\begin{proof}[Proof of Lemma \ref{RFDAEs:lemma00}]
Since the dimension of $\ker C(t)$ is constantly equal to $r>0$, by inspection of the proof 
of Theorem 3.9 of \cite[Chapter 3, \S 1]{dae} we get the existence of orthogonal matrix-valued 
functions $U, V\in C\big(\R, \R^{n\X n}\big)$ and $C_r\in C(\R, \R^{r\X r})$ such that, for all 
$t\in\R$, $\det C_r(t)\neq 0$ and
\begin{equation}\label{svdcont}
    U^T(t)C(t)V(t) = 
        \begin{pmatrix}
         C_r(t) & 0 \\ 
           0 & 0 
        \end{pmatrix}.
\end{equation}  

Let $U_r,V_r\in C\big(\R,\R^{n\X r}\big)$ and $U_0,V_0\in C(\R,\R^{n\X (n-r)})$ be matrix-valued
functions formed, respectively, by the first $r$ and $n-r$ columns of $U$ and $V$. 
A simple argument involving Equation \eqref{svdcont} shows that the columns of $V_0(t)$, $t\in\R$, 
are in $\ker C(t)$ and, since there are $n-r=\dim\ker C(t)$ of them, we have that the columns of 
$V_0(t)$ actually span $\ker C(t)$. In fact, the orthogonality of the matrix $V(t)$, $t\in\R$, 
imply that the columns of $V_0(t)$ form an orthogonal basis of $\ker C(t)$. A similar argument
proves that the columns of $U_0(t)$ are vectors of $\R^n$ that constitute an orthogonal basis of 
$\ker C(t)^T$ for all $t\in\R$. Observe also that since $\im C(t)$ is orthogonal to $\ker C(t)^T$ 
for all $t\in\R$, it follows that the columns of $U_r(t)$ are an orthogonal basis for $\im C(t)$ 
and that those of $V_r(t)$ so are for $\im C(t)^T$.

Similarly, let $P_r,Q_r$ and $P_0,Q_0$ be the matrices formed taking the first $r$ and $n-r$ 
columns of $P$ and $Q$, respectively. Since $P$ and $Q$ realize a singular value decomposition 
of $E$, one can check that the columns of $P_r$, $Q_r$, $P_0$ and $Q_0$ span $\im E$, $\im E^T$, 
$\ker E^T$, and $\ker E$, respectively.

We claim that $P_0^TU_r(t)$ is constantly the null matrix. To prove this, it is enough to 
show that for all $t\in\R$, the columns of $P_0$ are orthogonal to those of $U_r(t)$. Let $v$ 
and $u(t)$, $t\in\R$, be any column of $P_0$ and of $U_r(t)$, respectively. Since for all $t\in\R$ 
the columns of $U_r(t)$ are in $\im C(t)$, there is a vector $w(t)\in\R^n$ with the property that 
$u(t)=C(t)w(t)$, and 
\begin{equation*}
    \langle v, u(t)\rangle= \langle v, C(t)w(t)\rangle= 
    \langle C(t)^Tv, w(t)\rangle=0,\quad t\in\R,
\end{equation*}
because $v\in\ker E^T=\ker C(t)^T$ for all $t\in\R$. This proves the claim. A similar argument 
shows that also $P_r^TU_0(t)$, $V_r^T(t)Q_0$, and $V_0^T(t)Q_r$ are also identically zero.

Since for all $t\in\R$
\[
 P^TQ
=\begin{pmatrix}
                        P_r^TU_r(t) & 0 \\ 
                             0 & P_0^TU_0(t) 
                  \end{pmatrix}\quad
\text{ and }\quad
 V(t)^TQ
=\begin{pmatrix}
                        V_r(t)^TQ_r & 0 \\ 
                             0 & V_0(t)^TQ_0 
                  \end{pmatrix}
\]
are nonsingular, we deduce in particular that so are $P_r^TU_r(t)$ and $V_r(t)^TQ_r$.

Let us compute the matrix product $P^TC(t)Q$ for all $t\in\R$. We omit here, for the sake of 
simplicity, the explicit dependence on $t$. 
  \begin{equation*}
    \begin{aligned}
      P^TCQ &= P^TUU^TCVV^T Q
            =\begin{pmatrix}
                        P_r^TU_r & 0 \\ 
                             0 & P_0^TU_0 
                  \end{pmatrix}
        \begin{pmatrix}
                        C_r & 0 \\ 
                             0 & 0 
                  \end{pmatrix}\begin{pmatrix}
                        V_r^TQ_r & 0 \\ 
                             0 & V_0^TQ_0 
                  \end{pmatrix}\\
      &= \begin{pmatrix} P_r^TU_r C_r V_r^TQ_r & 0 \\
                                             0 & 0 
         \end{pmatrix}.
    \end{aligned}
  \end{equation*}
Which proves the assertion because  $P_r^TU_r$, $C_r$, and $V_r^TQ_r$ are nonsingular.
\end{proof}

\end{document}


\maketitle


%
\begin{abstract}
We present here a corrected version of Lemma 5.5 and Corollary 5.7 of EJQTDE 2011, No. 81.
\end{abstract}

\section {Introduction}
In Section 5 of \cite{BiSp} we investigated examples of applications of that paper's 
results to a particular class of implicit differential equations. For so doing we 
used a technical lemma from linear algebra that, unfortunately, turns out to be
flawed. As briefly discussed below this affects only marginally our paper's results
(just a corollary in Section 5 of \cite{BiSp}).

The simple example below shows that there is something wrong with Lemma 5.5 in 
\cite{BiSp}.  In the next section we provide an amended version of this result.

\begin{example}
 Consider the matrices
\[
 E=\begin{pmatrix}
    0 & 1\\
    0 & 0
   \end{pmatrix},
\quad
C=\begin{pmatrix}
   1 & 0\\
   0 & 0
  \end{pmatrix}.
\]
Clearly, $\ker C^T = \ker E^T= \spann\{\left(\begin{smallmatrix}0\\ 
                                                     1\end{smallmatrix}\right)\}$ for all 
$t\in\R$. The matrices
\[
 P=\begin{pmatrix}
    1 & 0\\
    0 & 1
   \end{pmatrix},
\quad
Q=\begin{pmatrix}
   0 & 1\\
   1 & 0
  \end{pmatrix},
\]
realize a singular value decomposition for $E$. Nevertheless
\[
 P^TCQ = \begin{pmatrix}
    0 & 1\\
    0 & 0
   \end{pmatrix}
\]
which is not the form expected from Lemma 5.5 in \cite{BiSp}. The problem, as it turns 
out, is that $\ker C \neq \ker E$.
\end{example}

Luckily, the impact of the wrong statement of \cite[Lemma 5.5]{BiSp} on \cite{BiSp} is 
minor: all results and examples (besides Lemma 5.5, of course) remain correct, with the 
exception of Corollary 5.7 where it is necessary to assume the following further 
hypothesis:
\begin{equation*}
   \ker\, C(t) = \ker\, E,\;\forall \,t\in\R.
  \end{equation*}
(A corrected statement of Corollary 5.7 of \cite{BiSp} can be found in the next section,
Corollary \ref{corcorretto}.)

\section{Corrected Lemma and its consequences}

We present here a corrected version of Lemma 5.5 in \cite{BiSp}. 

\begin{lemma} \label{RFDAEs:lemma00} 
Let $E\in\R^{n\X n}$ and $C\in C\big(\R, \R^{n\X n}\big)$ be respectively a matrix and a
matrix-valued function such that
\begin{equation}\label{RFDAEs:ex1a}
   \ker\, C^T(t) = \ker\, E^T,\;\forall \,t\in\R, \text{ and $\dim\ker\, E^T>0$},
  \end{equation}
Put $r=\rank\, E$, and let $P, Q\in \R^{n\X n}$ be orthogonal matrices that realize a 
singular value decomposition for $E$.  Then it follows that
  \begin{equation} \label{RFDAEs:continuousSVD} 
P^TC(t) Q
    = \begin{pmatrix} \widetilde{C}_{11}(t) & \widetilde{C}_{12}(t)  \\
      0 & 0 \end{pmatrix},\quad \forall t\in \R ,
  \end{equation}
with $\widetilde C_{11}\in C\big(\R, \R^{r\X r} \big)$ and 
$\widetilde C_{12}\in C\big(\R, \R^{r\X n} \big)$. 

If, furthermore,
\begin{equation}\label{RFDAEs:ex1b}
   \ker\, C(t) = \ker\, E,\;\forall \,t\in\R,
  \end{equation}
then  $\widetilde{C}_{12}(t)\equiv 0$. Namely, in this case,
\begin{equation} \label{RFDAEs:continuousSVD2} 
P^TC(t) Q
    = \begin{pmatrix} \widetilde{C}_{11}(t) & 0  \\
      0 & 0 \end{pmatrix},\quad \forall t\in \R ,
  \end{equation}
with $\widetilde{C}_{11}(t)$ nonsingular for all $t\in\R$.
\end{lemma}

\begin{proof}
Our proof is essentially a singular value decomposition (see, e.g., \cite{GoVL}) argument, 
based on a technical result from \cite{dae}. 

Observe that \eqref{RFDAEs:ex1a} imply $\rank E =\rank C(t)=r>0$ for all $t\in\R$.
In fact,
\begin{multline*}
\rank E=\rank E^T=n-\dim\ker E^T=\\
=n-\dim\ker C(t)^T =\rank C(t)^T=\rank C(t).
\end{multline*}
Since $\rank C(t)$ is constantly equal to $r>0$, by inspection of the proof of Theorem 3.9 
of \cite[Chapter 3, \S 1]{dae} we get the existence of orthogonal matrix-valued functions 
$U, V\in C\big(\R, \R^{n\X n}\big)$ and $C_r\in C(\R, \R^{r\X r})$ such that, for all 
$t\in\R$, $\det C_r(t)\neq 0$ and
\begin{equation}\label{svdcont}
    U^T(t)C(t)V(t) = 
        \begin{pmatrix}
         C_r(t) & 0 \\ 
           0 & 0 
        \end{pmatrix}.
\end{equation}  

Let $U_r,V_r\in C\big(\R,\R^{n\X r}\big)$ and $U_0,V_0\in C(\R,\R^{n\X (n-r)})$ be 
matrix-valued functions formed, respectively, by the first $r$ and $n-r$ columns of $U$ and 
$V$. An argument involving Equation \eqref{svdcont} shows that, for all $t\in\R$, the 
space $\im C(t)$ is spanned by the columns of $U_r(t)$. Also, \eqref{svdcont} imply that the 
columns of $V_0(t)$, $t\in\R$, belong to $\ker C(t)$ for all $t\in\R$. A dimensional argument 
shows that they constitute a basis $\ker C(t)$. Analogously, transposing  \eqref{svdcont}, 
we see that the columns of $V_r(t)$ and $U_0(t)$ are bases of $\im C(t)^T$ and $\ker C(t)^T$ 
respectively.\footnote{In fact, the orthogonality of the matrices $V(t)$ and $U(t)$ for all 
$t\in\R$, imply that the columns of $U_r(t)$, $V_r(t)$, $U_0(t)$ and $V_0(t)$ are respective 
orthogonal bases of the spaces $\im C(t)$, $\im C(t)^T$, $\ker C(t)^T$ and $\ker C(t)$.}

Let now $P_r,Q_r$ and $P_0,Q_0$ be the matrices formed taking the first $r$ and $n-r$ 
columns of $P$ and $Q$, respectively. Since $P$ and $Q$ realize a singular value 
decomposition of $E$, proceeding as above one can check that the columns of $P_r$, $Q_r$, 
$P_0$ and $Q_0$ span $\im E$, $\im E^T$, $\ker E^T$, and $\ker E$, respectively.

We claim that $P_0^TU_r(t)$ is constantly the null matrix in $\R^{(n-r)\X r}$. To prove this, 
it is enough to show that for all $t\in\R$, the columns of $P_0$ are orthogonal to those of 
$U_r(t)$. Let $v$ and $u(t)$, $t\in\R$, be any column of $P_0$ and of $U_r(t)$, respectively. 
Since for all $t\in\R$ the columns of $U_r(t)$ are in $\im C(t)$, there is a vector 
$w(t)\in\R^n$ with the property that $u(t)=C(t)w(t)$, and 
\begin{equation*}
    \langle v, u(t)\rangle= \langle v, C(t)w(t)\rangle= 
    \langle C(t)^Tv, w(t)\rangle=0,\quad \forall t\in\R,
\end{equation*}
because $v\in\ker E^T=\ker C(t)^T$ for all $t\in\R$. This proves the claim. A similar argument 
shows that $P_r^TU_0(t)$ is identically zero as well.

Since for all $t\in\R$
\[
 P^TU(t)
=\begin{pmatrix}
                        P_r^TU_r(t) & 0 \\ 
                             0 & P_0^TU_0(t) 
                  \end{pmatrix}\quad
\]
is nonsingular, we deduce in particular that so is $P_r^TU_r(t)$.

Let us compute the matrix product $P^TC(t)Q$ for all $t\in\R$. We omit here, for the sake of 
simplicity, the explicit dependence on $t$. 
  \begin{equation*}\label{PTCQ}
    \begin{aligned}
      P^TCQ &= P^TUU^TCVV^T Q
            =\begin{pmatrix}
                        P_r^TU_r & 0 \\ 
                             0 & P_0^TU_0 
                  \end{pmatrix}
        \begin{pmatrix}
                        C_r & 0 \\ 
                             0 & 0 
                  \end{pmatrix}\begin{pmatrix}
                        V_r^TQ_r & V_r^TQ_0 \\ 
                             V_0^TQ_r & V_0^TQ_0 
                  \end{pmatrix}\\
      &= \begin{pmatrix} P_r^TU_r C_r V_r^TQ_r & P_r^TU_rC_rV_r^TQ_0 \\
                                             0 & 0 
         \end{pmatrix},
    \end{aligned}
  \end{equation*}
which proves \eqref{RFDAEs:continuousSVD}.
\smallskip

Let us now assume that also \eqref{RFDAEs:ex1b} holds. We claim that in this case 
$V_0^TQ_r$ is identically zero. To see this we proceed as done above for the products 
$P_0^TU_r$ and $P_r^TU_0$. 
Let $v(t)$, $t\in\R$, be any column of $V_0(t)$, hence a vector of $\ker C(t)$ for 
all $t\in\R$, and let $q$ be a column of $Q_r(t)$. Since the columns of $Q_r$ lie in 
$\im E^T$, there is a vector $\ell\in\R^n$ with the property  that $q=E^T\ell$, and 
\begin{equation*}
    \langle v(t), q\rangle= \langle v(t), E^T\ell\rangle= 
    \langle E v(t), \ell\rangle=0,\quad \forall t\in\R,
\end{equation*}
because $v(t)\in\ker C(t)=\ker E$ for all $t\in\R$. This proves the claim. A similar 
argument shows that $V_r^TQ_0(t)$ is identically zero as well. Hence,
\[
 V(t)^TQ
=\begin{pmatrix}
                        V_r(t)^TQ_r & 0 \\ 
                             0 & V_0(t)^TQ_0 
                  \end{pmatrix}
\] 
thus $V_r^T(t)Q_0$, and $V_0^T(t)Q_r$ are nonsingular. Also, plugging $V_0^TQ_r=0$ in
the above expression for $P^TCQ$ one gets (we omit again the explicit dependence on $t$)
  \begin{equation}
    P^TCQ = \begin{pmatrix} P_r^TU_r C_r V_r^TQ_r & 0 \\
                                             0 & 0 
            \end{pmatrix}.
  \end{equation}
Which proves the assertion because  $P_r^TU_r$, $C_r$, and $V_r^TQ_r$ are nonsingular.
\end{proof}

In view of the corrected version of the above lemma, the statement of Corollary 5.7 of \cite{BiSp}
can be rewritten as follows:

\begin{corollary}\label{corcorretto}
Consider Equation
\begin{equation} \label{RFDAEs:exx1} 
E \dot{\xbf}(t) = \mathcal{F}\big(\xbf(t)\big)  + \lambda C(t) S(\xbf _t),
\end{equation}
where the maps $C\colon\R\to\R^{n\X n}$ and $S\colon\bun\to\R^n$ are continuous, $E$ is a 
(constant) $n\X n$ matrix, $\mathcal{F}$ is locally Lipschitz 
and $S$ verifies condition \hipK{} in \cite{BiSp}. Suppose also that $C$ and $E$ satisfy 
\eqref{RFDAEs:ex1a} and \eqref{RFDAEs:ex1b}, and that $C$ is $T$-periodic. Let $r>0$ be 
the rank of $E$ and assume that there exists an orthogonal basis of $\R^n\simeq\rrrn$ such that
$E$ has the form
\begin{equation*}
  E\simeq\begin{pmatrix} 
               E_{11} & E_{12} \\
               0      & 0 
        \end{pmatrix},\text{ with $E_{11}\in\R^{r\X r}$ invertible and 
         $E_{12}\in\R^{r\X (n-r)}$}. 
\end{equation*} 
Assume also that, relatively to this decomposition of $\R^n$, 
$\D_2{\mathcal{F}}_2(\xi,\eta)$ is invertible for all $x=(\xi,\eta)\in\rrrn$.

Let $\Omega$ be an open subset of $[0,+\infty)\X C_T(\R^n)$ and suppose that 
$\deg(\mathcal{F},\Omega\cap\R^n)$ is well-defined and nonzero. Then, there exists 
a connected subset $\Gamma$ of nontrivial $T$-periodic pairs for \eqref{RFDAEs:exx1} 
whose closure in $\Omega$ is noncompact and meets the set 
$\big\{(0,\su{\pbf})\in \Omega:\mathcal{F}(\pbf) =0\big\}$.
\end{corollary}

This result follows as in \cite{BiSp} taking into account the modified version of the lemma.